\documentclass[12pt]{amsart}

\usepackage{amsmath,amssymb,amsthm,amscd,euscript,longtable}
\def \r{\mathbb R}

\def \({\langle}
\def \){\rangle}


\DeclareMathOperator{\grad}{grad}

\DeclareMathOperator{\sech}{sech}\DeclareMathOperator{\csch}{csch}
\DeclareMathOperator{\Vol}{Vol}

\newtheorem{theorem}{Theorem}[section]
\newtheorem{lemma}[theorem]{Lemma}

\newtheorem{proposition}[theorem]{Proposition}
\newtheorem{corollary}[theorem]{Corollary}
\theoremstyle{remark}
\newtheorem{remark}[theorem]{Remark}
\theoremstyle{definition}
\newtheorem{definition}{Definition}
\newtheorem{example}[theorem]{Example}

\title{Mean value property for nonharmonic functions}
\author{Tetiana Boiko, Oleg Karpenkov}
\date{18 September 2013}

\address{Tetiana Boiko\\
TU Graz\\
30/III, Steyrergasse\\
8010, Graz, Austria} \email{boiko@math.tugraz.at}

\address{Oleg Karpenkov\\
University of Liverpool\\
Mathematical Sciences Building\\
Liverpool L69 7ZL, United Kingdom
} \email{karpenk@liv.ac.uk}

\keywords{Mean value property, Laplacian, discrete Laplacian, homogeneous trees}

\begin{document}

\input epsf

\begin{abstract}
In this article we extend the mean value property for harmonic functions to the nonharmonic case.
In order to get the value of the function at the center of a sphere
one should integrate a certain Laplace operator power series over the sphere.
We write explicitly such series in the Euclidean case and in the case of infinite homogeneous trees.
\end{abstract}

\maketitle

\tableofcontents

\section*{Introduction}
In this article we extend the mean value property for harmonic functions to the case of nonharmonic functions.
Our goals are to study this problem in Euclidean case and in the case of infinite homogeneous trees.

\vspace{2mm}

{\noindent
{\bf Mean value property for harmonic function.}
In what follows we denote by $S^{d-1}(r)$ the $(d{-}1)$-dimensional sphere in the Euclidean space $\r^d$
with radius $r$ and center at the origin.
Let $\Vol(S^{d-1}(x))$ be its volume and let $d\mu$ be the standard surface volume measure on each of the spheres.
}

Recall the classical {\it mean value property} for a harmonic function $f$:
$$
f(0)=\frac{1}{\Vol(S^{d-1}(r))}\int\limits_{S^{d-1}(r)} fd\mu.
$$
See~\cite{Helms1975} for general reference to potential theory.

\vspace{2mm}

{\noindent
{\bf Mean value property for nonharmonic function in Euclidean space $\r^d$.}

In Section~\ref{section-1} we prove the following general formula for analytic functions $f$ under some natural
convergency conditions (see Theorem~A):
$$
f(0)=\frac{1}{\Vol(S^{d-1}(r))}\int\limits_{S^{d-1}(r)}\sum\limits_{i=0}^\infty \alpha_{i,d}r^{2i}\triangle^i fd\mu.
$$
The coefficients  $\alpha_{i,d}$ are generated as follows
$$
\sum\limits_{i=0}^\infty\alpha_{i,d}x^{2i}=\frac{(Ix/2)^{\frac{d-2}{2}}}{\Gamma\big(\frac{d}{2}\big)J_{\frac{d-2}{2}}(Ix)},
$$
where $J_q$ denotes the Bessel function of the first kind and $I=\sqrt{-1}$.
}

We would like to mention that for harmonic functions all $\triangle^if=0$ (for $i\ge 1$) and hence we get a classical mean value property.

\vspace{2mm}

{\noindent
{\bf Mean value property for nonharmonic function  for homogeneous trees.}
Harmonic functions on trees for the first time were introduced in 1972 by P.~Cartier in~\cite{Cartier1972}.
In Section~\ref{section-2} of this article we show the generalized version of Poisson-Martin integral representation
for holomorphic functions to the case of non-harmonic functions (under certain natural convergency conditions).
For a general theory of harmonic functions on graphs and, in particular, trees we refer to~\cite{Doob2001,Dynkin1969,Woess2009}.}

Consider a homogeneous tree of degree $q+1$ which we denote by $T_{q}$.
We prove the following formula (see Theorem~B and Corollary~\ref{main-tree-cor}):

$$
f(v)=\frac{q}{q+1}\int\limits_{\partial T_q}
\left[
\sum\limits_{i=0}^{\infty}\Bigg((q+1)^i\Big(\gamma_i(\infty)+q^\infty\gamma_i(-\infty)\Big)\triangle^i f(t)\Bigg)\right]_v
dt,
$$
where
\begin{equation}\label{eq1}
\gamma_i(n)=c_{i,i}n^i+\ldots+ c_{i,1}n+c_{i,0},
\end{equation}
whose collection of coefficients $c_{i,j}$ (for a fixed $i$) is the solution of the following linear system
\begin{equation}\label{systemA}
A
\left(
\begin{array}{c}
c_{i,i}\\
\vdots\\
c_{i,1}\\
c_{i,0}\\
\end{array}
\right)
=
\left(
\begin{array}{c}
0\\
\vdots\\
0\\
1\\
\end{array}
\right)
\end{equation}
where
$$
A=\left(
\begin{array}{cccc}
0&\ldots  &0&2\\
1^i(1+(-1)^i q^1)& \ldots  &1^{1}(1+(-1)^{1} q^1)&1^0(1+(-1)^{0} q^1)\\
2^i(1+(-1)^i q^2)& \ldots  &2^{1}(1+(-1)^{1} q^2)&2^0(1+(-1)^{0} q^2)\\
\vdots&\ddots& \vdots&\vdots\\
(i{-}1)^i(1+(-1)^i q^{i-1})& \ldots  &(i{-}1)^{1}(1+(-1)^{1} q^{i-1})&(i{-}1)^{0}(1+(-1)^{0} q^{i-1})\\
i^i(1+(-1)^i q^{i})& \ldots  &i^{1}(1+(-1)^{1} q^{i})&i^{0}(1+(-1)^{0} q^{i})\\
\end{array}
\right).
$$
Here we consider the integration in the following sense
$$
\int\limits_{\partial T_q}
\left[\sum\limits_{i=0}^{\infty}\lambda_i(\infty)\triangle^i f(t)\right]_v
dt
=
\lim\limits_{n\to \infty} \left(\sum\limits_{i=0}^{n}\lambda_i(n)\sum\limits_{w\in S_n(v)}\triangle^i f(w)\right),
$$
where $S_n(v)$ is the set of all vertices at distance $n$ to the vertex $v$.

\section{Generalized mean value property in $\r^n$}\label{section-1}

In this section we we show how to generalize the mean value property in $\r^n$ to the case of nonharmonic functions.
Without loss of generality we study the value at the origin and take the integrals
over the spheres centered at the origin. In Subsections~\ref{1Lapl} and~\ref{1Bessel}
we introduce some preliminary general notions and definitions. Further in Subsection~\ref{1gvp}
we formulate and prove the main results concerning the mean value property in $\r^n$.

\subsection{Operator on $\r^1$ associated to the $d$-dimensional Laplace operator}\label{1Lapl}
Consider the Laplace operator $\triangle$ on $\r^d$.
In polar coordinates one can write
$$
\triangle(f)=\triangle_r(f)+\frac{1}{r^2}\triangle_{S^{d-1}}f, \quad \hbox{where} \quad\triangle_r(f)=
\frac{1}{r^{d-1}}\frac{\partial}{\partial r}\Big(r^{d-1}\frac{\partial f}{\partial r}\Big),
$$
the radial part, and $\triangle_{S^{d-1}}$ is the Laplace--Beltrami operator on the $(d{-}1)$-sphere.
Let us associate to the Laplace operator $\triangle$ the following operator on a real line:
$$
\tilde \triangle_d(g)=\frac{\partial}{\partial x}\Big(x^{d-1}\frac{\partial }{\partial x}\Big(\frac{g(x)}{x^{d-1}}\Big)\Big).
$$

\begin{proposition}\label{sphere-line}
For an analytic function $f$ it holds
$$
\int\limits_{S^{d-1}(x)}\triangle f(v)d\mu=\tilde\triangle_d\left(\int_{S^{d-1}(x)}f(v)d\mu\right).
$$
\end{proposition}

\begin{proof}
First, notice that for Laplace--Beltrami operator $\triangle_{S^{d-1}}$ it holds
$$
\int\limits_{S^{d-1}(x)}h(v)\triangle_{S^{d-1}} f(v)d\mu=-\int\limits_{S^{d-1}(x)}\langle \grad h(v),\grad f(v) \rangle d\mu,
$$
where the function $\grad$ is the gradient operator on the tangent space to the sphere $S^{d-1}(x)$, and $\langle v,w\rangle$
is the scalar product of $v$ and $w$.
Therefore, substituting $h=1$ we have
$$
\int\limits_{S^{d-1}(x)}\triangle_{S^{d-1}}f(v) {d\mu}=0.
$$

Second, we make the following transformations.
\begin{align*}
\tilde\triangle_d\Big(\int_{S^{d-1}(x)}f(v)d\mu\Big)&
{
=
\tilde\triangle_d\Big(x^{d-1}\int_{S^{d-1}(1)}f(xv)d\mu\Big)
=\frac{\partial}{\partial x}\left(x^{d-1}\frac{\partial }{\partial x} \Big(\int_{S^{d-1}(1)} f(xv)d\mu\Big)\right)
}
\\
&{=}
\int_{S^{d-1}(1)}\frac{\partial}{\partial x}\Big(x^{d-1}\frac{\partial}{\partial x} f(xv)\Big)d\mu
=\int_{S^{d-1}(x)}\frac{1}{x^d}\frac{\partial}{\partial x}\Big(x^{d-1}\frac{\partial}{\partial x} f(xv)\Big)d\mu
\\
&{=}\int_{S^{d-1}(x)}\triangle_r f(v)d\mu=\int_{S^{d-1}(x)}\triangle f(v)d\mu.
\end{align*}
This concludes the proof of Proposition~\ref{sphere-line}.
\end{proof}

Iteratively applying Proposition~\ref{sphere-line} we get the following corollary.
\begin{corollary}\label{sphere-line-n}
For an analytic function $f$ on $\r^d$ and a nonnegative integer $n$ it holds
$$
\int\limits_{S^{d-1}(x)}\triangle^n f(v)d\mu=\tilde\triangle_d^n\left(\int_{S^{d-1}(x)}f(v)d\mu\right).
$$
\qed
\end{corollary}


\subsection{Bessel functions and some important generating functions}\label{1Bessel}
Let $J_p$ denote Bessel functions of the first kind.
Recall that the power series decomposition of $J_p$ at $x=0$ is written as
$$
J_{p}(x)=\sum\limits_{k=0}^\infty \frac{(-1)^k(x/2)^{p+2k}}{k!\Gamma(p+k+1)}.
$$

Let us define two collections of coefficients $\alpha_{i,d}$ and $\beta_{i,d}$.
Recall that
$$
\sum\limits_{i=0}^\infty\alpha_{i,d}x^{2i}=\frac{(Ix/2)^{\frac{d-2}{2}}}{\Gamma\big(\frac{d}{2}\big)J_{\frac{d-2}{2}}(Ix)}.
$$
\begin{remark}
In case if $d=1$ and $d=3$ we have the following
$$
\sech x=\sum\limits_{i=0}^\infty\alpha_{i,1}x^{2i}  \qquad \hbox{and} \qquad x\csch x=\sum\limits_{i=0}^\infty\alpha_{i,3}x^{2i}.
$$
\end{remark}

Set the coefficients $\beta_{i,d}$ as follows
$$
\sum\limits_{i=0}^\infty\beta_{i,d}x^{2i}=\frac{J_{\frac{d-2}{2}}(Ix)}{(Ix/2)^{\frac{d-2}{2}}},
$$

\begin{proposition}\label{alpha-beta}
Let $k$ be a nonnegative integer and $d$ be a positive integer. Then it holds

\vspace{1mm}

$($i$)$ $\displaystyle \beta_{k,d}=\frac{1}{4^k k! \Gamma(p+k+1)}$;

\vspace{1mm}

$($ii$)$ $
\sum\limits_{i=0}^k \alpha_{i,d}\beta_{k-i,d}=
\left\{
\begin{array}{ll}
\frac{1}{\Gamma(d/2)}, &\hbox{if $k=0$;}\\
0,       &\hbox{if $k\ge 1$.}
\end{array}
\right.
$
\vspace{1mm}
\end{proposition}

\begin{proof}
The first statement follows directly from the power series decomposition for the function
$J_{\frac{d-2}{2}}(Ix)$.
The second statement holds, since by the definition of generating functions
$$
\sum\limits_{i=0}^\infty\alpha_{i,1}x^{2i}
\sum\limits_{i=0}^\infty\beta_{i,1}x^{2i}
=
\frac{(Ix/2)^{\frac{d-2}{2}}}{\Gamma\big(\frac{d}{2}\big)J_{\frac{d-2}{2}}(Ix)}
\cdot \frac{J_{\frac{d-2}{2}}(Ix)}{(Ix/2)^{\frac{d-2}{2}}}
=\frac{1}{\Gamma(d/2)}.
$$
\end{proof}

\subsection{Generalized mean value property}\label{1gvp}

We start with several definitions.

\begin{definition}
For an arbitrary nonnegative integer $d$ a smooth functions $f$ on $\r^n$, and a smooth function $g$ on $\r^1$ set
$$
\begin{array}{l}
\displaystyle
T_d(f,r)(v)=\sum\limits_{i=0}^\infty \alpha_{i,d}r^{2i}\triangle^i f(v),\\
\displaystyle
\tilde T_dg(x)=\sum\limits_{i=0}^\infty \alpha_{i,d}x^{2i}\tilde\triangle_d^ig(x),\
\end{array}
$$
where the generating function for the coefficients $\alpha_{i,d}$ is as above.
\end{definition}

For an arbitrary function $f:\r^d \to \r$ we denote by $\tilde f:\r \to \r$ the function defined as follows.
For positive $x$ we set
$$
\tilde f(x)=\frac{1}{\Vol(S^{d-1}(x))}\int_{S^{d-1}(x)}f(v)d\mu.
$$
For negative $x$ we put $f(x)=f(-x)$. Finally we define
$$
\tilde f(0)=\lim\limits_{x\to 0}\left(\frac{1}{\Vol(S^{d-1}(x))}\int_{S^{d-1}(x)}f(v)d\mu\right)=f(0).
$$

\begin{definition}
We say that a function $f$ is {\it spherically $a$-analytic at $0$} for some $a>0$
if the Taylor series for $\tilde f$ at the origin converges to $\tilde f$ on the segment $[-a,a]$.
\end{definition}

{
\noindent
{\bf Theorem A.}
{\it Consider $0<r<a$. Let $f:\r^d\to \r$ be a function that is spherically $a$-analytic at $0$.
Then we have
$$
f(0)=\frac{1}{\Vol(S^{d-1}(r))}\int\limits_{S^{d-1}(r)}T_d(f)d\mu.
$$
}
}

\begin{example}
Let a function $\varphi$ on $\r^3$ satisfy the Poisson's equation
$$
\triangle \varphi=f
$$
for some harmonic function $f$. Then it holds
$$
\varphi(0)=\frac{1}{4\pi}\int\limits_{S^{2}(1)}\Big(\varphi(x)-\frac{1}{6}\triangle \varphi(x)\Big)d\mu.
$$
\end{example}

We start the proof of Theorem~A with the following lemma.

\begin{lemma}\label{real-T-lemma}
Let $k$ be a nonnegative integer. Then
$$
\tilde T_d (x^{2k+d-1})=
\left\{
\begin{array}{ll}
x^{d-1}, &\hbox{if $k=0$;}\\
0,       &\hbox{if $n\ge d$.}
\end{array}
\right.
$$
\end{lemma}

\begin{proof}
First, observe the following
$$
\tilde\triangle _dx^n=(n-d+1)(n-1)x^{n-2}.
$$
Therefore,
$$
\tilde\triangle^i _dx^{2k+d-1}=4^{i} \frac{k!}{(k-i)!}\frac{\Gamma(k+\frac{d}{2})}{\Gamma(k-i+\frac{d}{2})}x^{n-2i}.
$$
In particular, this means that for $i>k$ we have $\tilde \triangle_d^i(x^{2k+d-1})=0$.
Hence we get
\begin{align*}
\tilde T_d(x^{2k+d-1})&=\sum\limits_{i=0}^\infty \alpha_{i,d}x^{2i}4^{i} \frac{k!}{(k-i)!}\frac{\Gamma(k+\frac{d}{2})}{\Gamma(k-i+\frac{d}{2})}x^{n-2i}\\
&
=4^k k!\Gamma\Big(k+\frac{d}{2}\Big)x^{2k+d-1}\sum\limits_{i=0}^k \alpha_{i,d}\frac{1}{4^{k-i}(k-i)!\Gamma(k-i+\frac{d}{2})}\\
&
=4^k k!\Gamma\Big(k+\frac{d}{2}\Big)x^{2k+d-1}\sum\limits_{i=0}^k \alpha_{i,d}\beta_{k-i,d}\\
&
=
\left\{
\begin{array}{ll}
x^{d-1}, &\hbox{if $k=0$;}\\
0,       &\hbox{if $k\ge 0$.}
\end{array}
\right.
\end{align*}
The last two equalities follows from Proposition~\ref{alpha-beta}(i) and  Proposition~\ref{alpha-beta}(ii) respectively.
\end{proof}

\begin{corollary}\label{tilda-relation}
Consider an even analytic function $f$ whose Taylor series taken at $0$ converges on the segment $[-a,a]$.
Let also $x$ satisfy $0<x<a$. Then
$$
\frac{\tilde T_d\big(x^{d-1}f(x)\big)}{x^{d-1}}=f(0).
$$
\end{corollary}

{\noindent
{\it Remark.} In fact, if $f$ is not even then a more general statement holds
$$
f(0)=\frac{\tilde T_d\big(x^{d-1}f(x)\big)+\tilde T_d\big(x^{d-1}f(-x)\big)}{2x^{d-1}}.
$$
}

\begin{proof}
Let $f$ be an even function, i.e.,
$$
f(x)=\sum\limits_{i=0}^\infty c_ix^{2i}.
$$
Then from Lemma~\ref{real-T-lemma} we have
$$
\frac{\tilde T_d\big(x^{d-1}f(x)\big)}{x^{d-1}}=
\frac{\tilde T_d\big(\sum_{i=0}^\infty c_ix^{2i+d-1}\big)}{x^{d-1}}=
\frac{\big(\sum_{i=0}^\infty c_i\tilde T_d(x^{2i+d-1})\big)}{x^{d-1}}=
\frac{c_0x^{d-1}}{x^{d-1}}=c_0=f(0).
$$
We demand the convergence of Taylor series in order to exchange the sum operation with $\tilde T_d$ in the second equality.
\end{proof}

{
\noindent
{\it Proof of Theorem~A.}
By Corollary~\ref{sphere-line-n} and by the definition of $\tilde f$ we have
$$
\int\limits_{S^{d-1}(x)}\triangle^n f(v)d\mu=\tilde\triangle_d^n\left(\int_{S^{d-1}(x)}f(v)d\mu\right)=
\Vol(S^{d-1}(1))\tilde\triangle_d^n \big(x^{d-1}\tilde f(x)\big).
$$
}

Since $f$ is spherically $a$-analytic, the function $\tilde f$ satisfies all the conditions of Corollary~\ref{tilda-relation}.
Applying Corollary~\ref{tilda-relation} we get
\begin{align*}
f(0)&=\tilde f(0)=\frac{\tilde T_d\big(x^{d-1} \tilde f(x)\big)}{x^{d-1}}=
\frac{1}{x^{d-1}\Vol(S^{d-1}(1))}\int\limits_{S^{d-1}(x)}T_d(f)d\mu\\
&=
\frac{1}{\Vol(S^{d-1}(x))}\int\limits_{S^{d-1}(x)}T_d(f)d\mu.
\end{align*}
This concludes the proof of Theorem~A. \qed


\section{Horocyclic formula for homogeneous trees}\label{section-2}

In this section we study the situation in the discrete case of homogeneous trees.
We start in Subsection~\ref{2n&d} with necessary notions and definitions.
Further in Subsection~\ref{2main} we formulate the statements regarding the generalization of the
Poisson-Martin integral representation theorem.
In Subsection~\ref{2relations} we study some necessary tools that are further used in the proofs of the main result.
We conclude the proofs in Subsection~\ref{2conclude}.

\subsection{Notions and definitions}\label{2n&d}
Consider a homogeneous tree $T_{q}$ (i.e., every vertex of such tree has $q+1$ neighbors) and denote its Martin boundary by $\partial T_{q}$.
If $v$ and $w$ are connected by an edge we write $v\sim w$.

\subsubsection{Laplace operator}
In this section we consider the standard Laplace operator on the space of all functions on $T_q$, which is defined as
$$
\triangle f(v) =\frac{\sum\limits_{w\sim v} f(w)}{q+1}-f(v).
$$
The composition of $i\ge 1$ Laplace operators we denote by $\triangle^i$. Set also $\triangle^0$ the identity operator.

\begin{remark}
Similarly one might consider majority of weighted Laplace operators. The statements of this section has a straightforward generalization
to arbitrary locally finite graphs. For simplicity reasons we restrict ourselves entirely to homogeneous trees.
\end{remark}

\subsubsection{Maximal cones and horocycles}
We start with the definition of maximal proper cones.

\begin{definition}
Consider two vertices $v,w\in T_{q}$ connected by an edge $e$. The maximal connected component
of $T_{q}\setminus e$ containing $v$ is called the {\it maximal proper cone with vertex at $v$ $($with respect to w$)$}. We denote it by $C^{v-w}$.
\end{definition}

The {\it distance} between two vertices $v,w\in T_{q}$ is the minimal number of edges
needed to reach the vertex $w$ starting from the vertex $v$.
For an arbitrary nonnegative integer $r$ and an arbitrary vertex $v$ we denote by $S_r(v)$ the set of all vertices at distance $r$ to $v$, we call such
set the {\it circle of radius $r$ with center $v$}. Note that $S_r(v)$ contains exactly $(q+1)q^{r-1}$ points.

\begin{definition}
Let $C^{v-w}$ be a maximal proper cone of $T_q$ and let $n$ be a nonnegative integer. The set
$$
C^{v-w}_n=C_{v-w}\cap S_n(v)
$$
is called the {\it horocycle of radius $n$ with center at $v$ $($with respect to $w$$)$}.
\end{definition}

\subsubsection{Integral series}
For an arbitrary function $f:T_q\to \r$ we write
$$
f(C^{v-w}_n)=\frac{1}{q^n}\sum\limits_{u\in C^{v-w}_n} f(u).
$$
%

\begin{definition}
In what follows we consider the {\it horocyclic integrals} defined by the following expression:
$$
\int\limits_{\partial C^{v-w}}
\left[\sum\limits_{i=0}^{\infty}\lambda_i(\infty)\triangle^i f(t)\right]
dt
=
\lim\limits_{n\to \infty} \left(\sum\limits_{i=0}^{n}\lambda_i(n)\triangle^i f(C^{v-w}_{n})\right),
$$
where $f$ is a function on the tree, $\lambda_i$ are arbitrary functions on the set of positive integers.
Respectively we write
$$
\int\limits_{\partial T_q}
\left[\sum\limits_{i=0}^{\infty}\lambda_i(\infty)\triangle^i f(t)\right]_v
dt
=
\lim\limits_{n\to \infty} \left(\sum\limits_{i=0}^{n}\lambda_i(n)\sum\limits_{u\in S^n(v)}\frac{\triangle^i f(u)}{q^n}\right).
$$
Here we specify by an index $v$ that the series are taken with respect to the vertex $v$, since now it is not reconstructed
from the integration domain.
\end{definition}

For instance,
$$
\int\limits_{\partial C^{v-w}}
[2^\infty(1+(-\infty)^3)f(t)]
dt
=
\lim\limits_{n\to \infty} \Big(2^n(1-n^3)f(C^{v-w}_{n+1})\Big)
=
\lim\limits_{n\to \infty} \Big(2^n(1-n^3)\sum\limits_{u \in C^{v-w}_{n}}\frac{f(u)}{q^n}\Big).
$$

\vspace{2mm}

{
\noindent
{\it Remark.}
Notice that the limit operation is not always commute with the sum operation.
To illustrate this we mention, that the expression from the limit exists for every holomorphic function even if the integral at Martin boundary diverges
(see Theorem~\ref{holom}).
So the notion of integral series extends the notion of integration of functions at Martin boundary.
}

\subsection{Horocyclic formula}\label{2main}
In this subsection we formulate the mean value property for certain nonharmonic functions.

\subsubsection{Horocyclic integrals for horosummable functions}

We start with the following definition.
\begin{definition}
We say that a functions $f$ is {\it $C^{v-w}$-horosummable} if
$$
\lim\limits_{n\to \infty} \big(q^nf(C^{v-w}_{2n})\big)=
\lim\limits_{n\to \infty} \Big(\sum\limits_{u\in C^{v-w}_{2n}} \frac{f(u)}{q^n}\Big)=0.
$$
\end{definition}

{
\noindent
{\bf Theorem B.}
{\it
Consider two vertices $v,w\in T_q$ connected by an edge, and let
$f$ be a $C^{v-w}$-horosumable function.
Then
$$
f(v)=\int\limits_{\partial C^{v-w}}
\left[
\sum\limits_{i=0}^{\infty}\Bigg((q+1)^i\Big(\gamma_i(\infty)+q^\infty\gamma_i(-\infty)\Big)\triangle^i f(t)\Bigg)\right]
dt,
$$
where $\gamma_i$ are polynomials as below $($see~$($\ref{eq1}$)$$)$ whose coefficients are the solutions of System~$($\ref{systemA}$)$.
In addition, the condition that the horocyclic integral in the right part of the equation converges to $f(v)$
is equivalent to the condition that $f$ is $C^{v-w}$-horosumable.
}
}

\vspace{2mm}

We prove this theorem later in Subsection~\ref{2conclude}.

Note that it would be interesting to relate the coefficients at terms $\triangle^i(f)$ with discretizations of Bessel functions.

\begin{example}
Let us check Theorem~B for the function $\chi_v$ that is zero everywhere except for the point $v$ and $\chi_v(v)=1$.
We have
$$
\triangle^i\chi_v(C_n^{v-w})=\frac{1}{q^n}\sum\limits_{u\in C^{v-w}_n} f(u)=
\left\{
\begin{array}{ll}
0, &\hbox{if $i<n$;}\\
\frac{1}{(q+1)^n},       &\hbox{if $i=n$.}
\end{array}
\right.
$$
(notice that $C^{v-w}_n$ contains exactly $q^n$ vertices).
Therefore,
$$
\begin{array}{l}
\displaystyle
\int\limits_{\partial C^{v-w}}
\left[
\sum\limits_{i=0}^{\infty}\Bigg((q+1)^i\Big(\gamma_i(\infty)+q^\infty\gamma_i(-\infty)\Big)\triangle^i \chi_v(t)\Bigg)\right]
dt
\\
\displaystyle\qquad \qquad=
\lim\limits_{n\to\infty}a_{n,n}\triangle^n\chi_v(C^{v-w}_n)=
\lim\limits_{n\to\infty}(q+1)^n\frac{1}{(q+1)^n}=1=\chi_v(v).
\end{array}
$$
It is clear from this example that it is not always possible to exchange the sum operator and the limit operator.
For the function $\chi_v$ we have
$$
\sum\limits_{i=0}^{\infty}
\lim\limits_{n\to \infty}
\Bigg((q+1)^i\Big(\gamma_i(n)+q^n\gamma_i(-n)\Big)\triangle^i \chi_v(C_n^{v-w})\Bigg)=
\sum\limits_{i=0}^{\infty}0=0\ne 1=\chi_v(v).
$$
\end{example}

Let us write a weaker version of Theorem~B for the integration over all the Martin boundary.

\begin{corollary}\label{main-tree-cor}
Consider a vertex $v\in T_q$, and let
$f$ be a $C^{v-w}$-horosumable function for all vertices $w$ adjacent to $v$.
Then
$$
f(v)=\frac{q}{q+1}\int\limits_{\partial T_q}
\left[
\sum\limits_{i=0}^{\infty}\Bigg((q+1)^i\Big(\gamma_i(\infty)+q^\infty\gamma_i(-\infty)\Big)\triangle^i f(t)\Bigg)\right]_v
dt.
$$
\end{corollary}

\begin{proof}
Let us sum up the expression obtained in Theorem~B for all maximal proper cones with vertex at $v$.
From one hand there are exactly $q{+}1$ such horocycles so the sum equals to $(q+1)f(v)$.
From the other hand each point of the Martin boundary was integrated $q$ times. Therefore, we get the constant $\frac{q}{q+1}$
in the statement of the corollary.
\end{proof}

{\noindent
{\it Remark.}
Note that it is possible to write similar series for arbitrary locally-finite trees, although the formulas for the coefficients
would be more complicated.
}

\subsubsection{Horocycle formula for harmonic functions}
We conclude this subsection with the following more general statement for harmonic functions.

\begin{corollary}\label{holom}
Consider an arbitrary harmonic function $h$ on a homogeneous tree $T_q$.
Let $v$ be a vertex of $T_q$ and $G_v$ be one of the corresponding horocyclic parts.
Then the following holds:
$$
h(v)=\int\limits_{\partial C^{v-w}}[h(t)]dt+
\int\limits_{\partial C^{v-w}}\bigg[q^\infty \Big(h(t)-\int_{\partial C^{v-w}}[h(t)]dt\Big)\bigg]dt.
$$
\qed
\end{corollary}

{
\noindent
{\it Remark.}
Suppose that $h$ is integrable on $\partial C^{v-w}$ with respect to the probability measure on the Martin boundary.
Then this integral coincides with
$$
\int\limits_{\partial C^{v-w}}[h(t)]dt.
$$
In case if $h$ is not integrable with respect to probability measure,
the horocyclic integral nevertheless exists. In some sense horocyclic integrability is a conditional integrability
with respect to integration over probability measure.
Horocyclic integral exists for every harmonic function $h$ and for every cone $C^{v-w}$.
}

\subsection{Relations on special Laurent polynomial}\label{2relations}
In this subsection we prove some supplementary statements.
For every integer $n$ we denote
$$
D_n(x)=x^n+\frac{q^n}{x^n}.
$$
Note that $D_0(x)=x^0+\frac{q^0}{x^0}$.

For every nonnegative integer we set
$$
S_n(x)=\frac{(x-1)^n(x-q)^n}{(q+1)^nx^n}.
$$

\vspace{2mm}

We have the following recurrent relation for the defined above Laurent polynomials.
\begin{proposition}\label{recursion}
For every integer $n$ we have
$$
S_1D_n=\frac{D_{n+1}-(q+1)D_n+qD_{n-1}}{q+1}.
$$
\end{proposition}

\begin{proof}
For every integer $n$ (including $n=-1,0,1$) it holds
\begin{align*}
S_1D_n&=\Big(\frac{(x-1)(x-q)}{(q+1)x}\Big)\Big(x^n+\frac{q^n}{x^n}\Big)\\
&=\frac{x^{n+1}}{q+1}-x^n+\frac{q}{q+1}x^{n-1}+
\frac{q^n}{(q+1)x^{n-1}}-\frac{q^n}{x^n}+\frac{q^{n+1}}{(q+1)x^{n+1}}\\
&=\frac{1}{q+1}\Big(x^{n+1}+\frac{q^{n+1}}{x^{n+1}}\Big)-
\Big(x^n+\frac{q^n}{x^n}\Big)+
\frac{q}{q+1}\Big(x^{n-1}+\frac{q^{n-1}}{x^{n-1}}\Big)\\
&=\frac{D_{n+1}-(q+1)D_n+qD_{n-1}}{q+1}.
\end{align*}
\end{proof}

The following proposition is straightforward.
\begin{proposition}\label{uniqueness_SD}
For every integer $n$ there exists a unique decomposition
$$
D_n=\sum\limits_{i=0}^n a_{n,i}S_i.
$$
\qed
\end{proposition}

Now we are interested in the coefficients $a_{n,i}$.
The next statement follows directly from Proposition~\ref{recursion}.

\begin{corollary}\label{cor_rec}
For every positive integer $i$ and every integer $n$ it holds
$$
a_{n,i-1}=\frac{a_{n+1,i}-(q+1)a_{n,i}+qa_{n-1,i}}{q+1}.
$$
Additionally in the case $i=0$ it holds
$$
0=a_{n+1,0}-(q+1)a_{n,0}+qa_{n-1,0}.
$$
\end{corollary}

\begin{proof}
By the definition we have
$$
S_1S_k=S_{k+1}.
$$
Propositions~\ref{recursion} and~\ref{uniqueness_SD} imply
\begin{align*}
\sum\limits_{i=1}^{n+1} a_{n,{i-1}}S_{i}&=S_1D_n=\frac{D_{n+1}-(q+1)D_n+qD_{n-1}}{q+1}\\
&= \frac{1}{q+1}\Big(\sum\limits_{i=0}^{n+1} a_{n+1,i}S_{i}-(q+1)
\sum\limits_{i=0}^n a_{n,i}S_{i}+
q\sum\limits_{i=0}^{n-1} a_{{n-1},i}S_{i}\Big).
\end{align*}
Collecting the coefficients at $S_i$ we get the recurrence relations of the corollary.
\end{proof}

\begin{definition}
For a positive integer $k$ we define the linear form $L_k$ in $2k{+}1$ variables as follows
$$
L_k(y_1,\ldots,y_{2k+1})=\sum\limits_{i=-n}^{n}c_{i,n}y_i,
$$
where $c_{i,n}$ are defined as the coefficients of $S_n$, i.e., from the expression
$$
S_n(x)=\frac{(x-1)^n(x-q)^n}{(q+1)^nx^n}=\sum\limits_{i=-n}^{n}c_{i,n}x^i.
$$
\end{definition}

\begin{proposition}
For every nonnegative integer $i$ and every integer $n$ we have
$$
L_i(a_{n-i,i},a_{n-i+1,i},\ldots,a_{n+i,i})=0.
$$
\end{proposition}

\begin{proof}
We prove the proposition by induction in $i$.

\vspace{2mm}

{\noindent
{\it Base of induction.} For the case $i=0$ the statement holds by Corollary~\ref{cor_rec}.
}

\vspace{2mm}

{\noindent
{\it Step of induction.} Suppose that the statement holds for $i-1$. Let us prove it for $i$. We have
$$
L_i(a_{n-i,i},\ldots,a_{n+i,i})=0.
$$
By Corollary~\ref{cor_rec} and linearity of $L_i$ we have
$$
\begin{array}{l}
L_{i}(a_{n-i,i},\ldots,a_{n+i,i})\\
\quad=L_{i}\big(\frac{a_{n-i+1,i+1}-(q{+}1)a_{n-i,i+1}+qa_{n-i-1,i+1}}{q+1},\ldots,\frac{a_{n+i+1,i+1}-(q{+}1)a_{n-i,i+1}+qa_{n-i-1,i+1}}{q+1}\big)\\
\displaystyle
\quad=\frac{1}{q+1}\Big(L_{i}(a_{n-i+1,i+1},\ldots,a_{n+i+1,i+1})-(q+1)L_{i}(a_{n-i,i+1},\ldots,a_{n+i,i+1})\\
\qquad + \hbox{ }qL_{i}(a_{n-i-1,i+1},\ldots,a_{n+i-1,i+1})\Big)\\
\quad=L_{i+1}(a_{n-i-1,i+1},a_{n-i,i+1},\ldots,a_{n+i,i+1},a_{n+i+1,i+1}).
\end{array}
$$
Therefore, by induction assumption we have
$$
L_{i+1}(a_{n-i-1,i+1},a_{n-i,i+1},\ldots,a_{n+i,i+1},a_{n+i+1,i+1})=L_{i}(a_{n-i,i},\ldots,a_{n+i,i})=0.
$$
This concludes the proof of the induction step.}
\end{proof}

\begin{corollary}\label{degree}
For every fixed nonnegative integer $k$ we have the
$$
a_{n,k}=P_k(n)+q^n\hat P_k(n),
$$
where $P_k(n)$ and $\hat P_k(n)$ are polynomials of degree at most $k$.
\qed
\end{corollary}

We skip the proof here. This is a general statement about linear recursive
sequences whose characteristic polynomial has roots $1$ and $q$ both of multiplicity $n$.

\begin{example}
Direct calculations show that in case $q=2$ we have
$$
\begin{array}{l}
a_{n,0}=1+2^n,\\
a_{n,1}=\frac{3^1}{1!}(-n+2^nn),\\
a_{n,2}=\frac{3^2}{2!}\big(n^2+3n+2^n(n^2-3n)\big),\\
a_{n,3}=\frac{3^3}{3!}\big(-n^3-9n^2-26n+2^n(n^3-9n^2+26n)\big),\\
\ldots
\end{array}
$$
\end{example}

Let us prove a general theorem on numbers $a_{n,i}$.

\begin{theorem}\label{polynomials}
For every admissible $k$ and $n$ it holds
$$
a_{n,k}=(q+1)^k\big(\gamma_k(n)+q^n\gamma_k(-n)\big),
$$
where the coefficients of $\gamma_k$ are defined by System~$($\ref{systemA}$)$.
\end{theorem}


We start the proof of Theorem~\ref{polynomials} with the following two lemmas.

\begin{lemma}\label{p=q}
For every nonnegative integer $k$ and every $n$ we have
$$
\hat P_k(-n)=P_k(n).
$$
\end{lemma}

\begin{proof}
For every integer $x$ we have
$$
D_{-n}=x^{-n}+\frac{q^{-n}}{x^{-n}}=\frac{1}{q^{n}}\Big(\frac{q^{n}}{x^n}+x^n\Big)=\frac{D_n}{q^n}.
$$
By Proposition~\ref{uniqueness_SD} the coefficients $a_{n,i}$ and $a_{-n,i}$ are uniquely defined, therefore,
$$
a_{n,k}=q^n a_{-n,k}.
$$
Let us rewrite this equality in terms of polynomials $P_k$ and $\hat P_k$:
$$
P_k(n)+q^n\hat P_k(n)=q^n(P_k(-n)+q^{-n}\hat P_k(-n)),
$$
and hence
$$
P_k(n)+q^n\hat P_k(n)=\hat P_k(-n)+q^nP_k(-n).
$$
Since this equality is fulfilled for every $n$ we have $\hat P_k(-n)=P_k(n)$. This concludes the proof.
\end{proof}

\begin{lemma}\label{2lem2}
For every nonnegative $k$ it holds
$$
P_k(k)+q^kP_k(-k)=(q+1)^k
$$
\end{lemma}

\begin{proof}
We prove the proposition by induction in $k$.

\vspace{2mm}

{\noindent
{\it Base of induction.} For the case $k=0,1$ we have
$$
P_0(0)+q^0P_0(0)=a_{0,0}=1 \qquad \hbox{and} \qquad P_1(1)+qP_1(-1)=a_{1,1}=q+1.
$$
}

\vspace{2mm}

{\noindent
{\it Step of induction.} Let $P_k(k)+q^kP_k(-k)=(q+1)^k$. Then
\begin{align*}
(q+1)^k&=P_{k}(k)+q^{k}P_{k}(-k)=a_{k,k}=\frac{a_{k+1,k+1}-(q+1)a_{k,k+1}+qa_{k-1,k+1}}{q+1}
=\frac{a_{k+1,k+1}}{q+1}\\
&=\frac{P_{k+1}(k{+}1)+q^{k+1}P_{k+1}({-}k{-}1)}{q+1}.
\end{align*}
The third equality follows from the recursive formula of Corollary~\ref{cor_rec}.
Hence
$$
P_{k+1}(k{+}1)+q^{k+1}P_{k+1}({-}k{-}1)=(q+1)^{k+1}.
$$
This concludes the step of induction.
}
\end{proof}

\vspace{2mm}

{\noindent
{\it Proof of Theorem~\ref{polynomials}.}
From Lemma~\ref{p=q} we know that $\hat P_k(-n)=P_k(n)$. In addition, by Corollary~\ref{degree}
the degree of $P_k$ equals to $k$, and hence it has $k+1$ coefficient.
The coefficients of the polynomial $P_k$ are uniquely defined by the conditions for $a_{j,k}$ for $j=0,\ldots, k$:
$$
\hbox{$P_k(j)+q^jP_k(-n)=0$ for $j=0,\ldots,{k-1}$}, \quad \hbox{and} \quad P_k(k)+q^kP_k(-k)=(q+1)^k.
$$
The expression for $k$ follows from Lemma~\ref{2lem2}.
We consider these equalities as linear conditions on the coefficients of the polynomial $\frac{P_k}{(q+1)^k}$.
These conditions form a linear system, which coincides with System~$($\ref{systemA}$)$ (substituting $k$ to $i$).
}

\vspace{2mm}

We should also show that the determinant of the matrix in System~(\ref{systemA}) is nonzero.
We prove this by reductio ad absurdum. Suppose the determinant of the matrix is zero.
Thus, it has a nonzero kernel. Therefore, there exists an expression
$$
R(n)=r(n)+r(-n)q^n,
$$
where $r(n)$ is a polynomial of degree $k$ having at least one nonzero coefficient, satisfying
$$
R(-k)=R(-k+1)=\ldots=R(k)=0.
$$
Let $R(k+1)=a$. Let us find the value $R(-k-1)$. From one hand, our sequence satisfy the linear recursion condition
determined by the coefficients of the polynomial $(x-1)^k(x-q)^k$, and hence
$$
R(-k-1)=-\frac{a}{q^{k+1}}.
$$
From another hand,
$$
R(-k-1)=r(-k-1)+r(k+1)q^{-k-1}=\frac{r(k+1)+r(-k-1)q^{k+1}}{q^{k+1}}=\frac{a}{q^{k+1}}.
$$
This implies that $a=0$, and hence $R(k+1)=R(-k-1)=0$.

Therefore, the linear recursive sequence $R(n)$ determined by the coefficients of the polynomial of degree ${2k+3}$
has $2k+3$ consequent elements equal zero.
Hence $R(n)=0$ for any integer $n$, which implies that all the coefficients of $r(n)$ equal zero.
We come to the contradiction. Hence the determinant of the matrix in System~(\ref{systemA}) is nonzero.

\vspace{2mm}

So both the coefficients of $\frac{P_k}{(q+1)^k}$ and the coefficients of $\gamma_k$ are solutions of System~$($\ref{systemA}$)$.
Since System~$($\ref{systemA}$)$ has a unique solution, the polynomials $P_k$ and $(q+1)^k\gamma_k$ coincide.
Therefore, by Lemma~\ref{p=q} it holds
$$
a_{n,k}=P_k(n)+q^n\hat P_k(n)=P_k(n)+q^nP_k(-n)=(q+1)^k\big(\gamma_k(n)+q^n\gamma_k(-n)\big).
$$
This concludes the proof of Theorem~\ref{polynomials}.
\qed

Observe the following corollary.
\begin{corollary}
For every integer $k>0$ we have $P_k(0)=0$, and $P_0(1)=1$.
\qed
\end{corollary}

\subsection{Proof of Theorem~B}\label{2conclude}
Finally we have all necessary tools to prove of Theorem~B.
We start with the following lemma.

\begin{lemma}\label{expression}
Let $f$ be a function on $T_q$ and $v,w$ be two vertices of $T_q$ connected by an edge.
Then for every nonnegative $n$ it holds
$$
f(v)+q^nf(C^{v-w}_{2n})=\sum\limits_{k=0}^{n}\Big((q+1)^k\big(\gamma_k(n)+q^n\gamma_k(-n)\big)\triangle^k f(C^{v-w}_{n})\Big).
$$
\end{lemma}

\begin{proof}
For $0<k\le n$ set
$$
\begin{array}{l}
\hat D_{k,n}=f(C^{v-w}_{n-k})+q^kf(C^{v-w}_{n+k}),\\
\hat S_{k,n}=\sum\limits_{i=-k}^kc_{i,k}f(C^{v-w}_{n+i}),
\end{array}
$$
where the coefficients $c_{i,k}$ are generated by
$$
S_k=\frac{\big((x-1)(x-q)\big)^k}{(q+1)^kx^k}=\sum\limits_{i=-k}^k c_{i,k}x^i.
$$
Notice that all linear expressions over $S_k$ and $D_k$ are identically translated to the linear expressions over $\hat S_{k,n}$ and $\hat D_{k,n}$.
Then from Proposition~\ref{uniqueness_SD} it follows
$$
f(v)+q^nf(C^{v-w}_{2n})=\hat D_{n,n}=\sum\limits_{k=0}^n a_{n,k}\hat S_{k,n},
$$
where the coefficients $a_{n,k}$ as in Theorem~\ref{polynomials}, i.e.,
$$
a_{n,k}=(q+1)^k(\gamma_k(n)+q^n\gamma_k(-n)),
$$
where the coefficients of $\gamma_k$ are defined by System~$($\ref{systemA}$)$.
In addition note that
$$
\hat S_{k,n}=\triangle^k(C_n^{v-w}).
$$
Therefore, we obtain
$$
f(v)+q^nf(C^{v-w}_{2n})=\sum\limits_{k=0}^{n}\Big((q+1)^k\big(\gamma_k(n)+q^n\gamma_k(-n)\big)\triangle^k f(C^{v-w}_{n})\Big).
$$
This concludes the proof.
\end{proof}

\vspace{2mm}

{\noindent
{\it Proof of Theorem~B.} From Lemma~\ref{expression} we have
$$
f(v)+q^nf(C^{v-w}_{2n})=\sum\limits_{i=0}^{n}\Big((q+1)^i\big(\gamma_i(n)+q^n\gamma_i(-n)\big)\triangle^i f(C^{v-w}_{n})\Big).
$$
Hence,
}
$$
\begin{array}{l}
\displaystyle
\int\limits_{\partial C^{v-w}}
\left[
\sum\limits_{i=0}^{\infty}\Bigg((q+1)^i\Big(\gamma_i(\infty)+q^\infty\gamma_i(-\infty)\Big)\triangle^i f(t)\Bigg)\right]
dt
\\
\displaystyle
\qquad \qquad =
\lim\limits_{n\to\infty}\sum\limits_{i=0}^{n}\Bigg((q+1)^i\Big(\gamma_i(n)+q^n\gamma_i(-n)\Big)\triangle^i f(C^{v-w}_{n})\Bigg)
\\
\qquad \qquad =
\lim\limits_{n\to\infty}\big(f(v)+q^nf(C^{v-w}_{2n})\big)
=f(v)+\lim\limits_{n\to\infty}\big(q^nf(C^{v-w}_{2n})\big)
\\
\qquad \qquad
=f(v).
\end{array}
$$
Therefore, the integral converges to $f(v)$ if and only if the sequence $\big(q^nf(C^{v-w}_{2n})\big)$ converges to zero
as $n$ tends to infinity. This means that $f$ is $C^{v-w}$-horosumable.
This concludes the proof.
\qed

\vspace{2mm}

{\noindent
{\bf Acknowledgements.}
The authors are grateful to Prof.~Woess for constant attention to this work.
This first author is supported by the Austrian Science Fund (FWF):
W1230, Doctoral Program ``Discrete Mathematics''.
}

\bibliographystyle{plain}

\begin{thebibliography}{1}

\bibitem{Cartier1972}
P.~Cartier.
\newblock Fonctions harmoniques sur un arbre.
\newblock In {\em Symposia {M}athematica, {V}ol. {IX} ({C}onvegno di {C}alcolo
  delle {P}robabilit\`a, {INDAM}, {R}ome, 1971)}, pages 203--270. Academic
  Press, London, 1972.

\bibitem{Doob2001}
J.~L. Doob.
\newblock {\em Classical potential theory and its probabilistic counterpart}.
\newblock Classics in Mathematics. Springer-Verlag, Berlin, 2001.
\newblock Reprint of the 1984 edition.

\bibitem{Dynkin1969}
E.~B. Dynkin.
\newblock The boundary theory of {M}arkov processes (discrete case).
\newblock {\em Uspehi Mat. Nauk}, 24(2 (146)):3--42, 1969.

\bibitem{Helms1975}
L.~L. Helms.
\newblock {\em Introduction to potential theory}.
\newblock Robert E. Krieger Publishing Co., Huntington, N.Y., 1975.
\newblock Reprint of the 1969 edition, Pure and Applied Mathematics, Vol. XXII.

\bibitem{Woess2009}
W.~Woess.
\newblock {\em Denumerable {M}arkov chains}.
\newblock EMS Textbooks in Mathematics. European Mathematical Society (EMS),
  Z\"urich, 2009.
\newblock Generating functions, boundary theory, random walks on trees.

\end{thebibliography}

\vspace{.5cm}
\end{document}